\documentclass[12pt,psamsfonts]{amsart}
\usepackage{amssymb,amsmath,amscd,eucal,epsfig}
\usepackage{tikz}

\def\Dj{\hbox{D\kern-.73em\raise.30ex\hbox{-} \raise-.30ex\hbox{}}}
\def\dj{\hbox{d\kern-.33em\raise.80ex\hbox{-} \raise-.80ex\hbox{\kern-.40em}}}

\def\<{\langle}                     
\def\>{\rangle}                     

\newtheorem{thm}{Theorem}[section]

\newtheorem{conj}[thm]{Conjecture}

\newcommand{\ben}{\begin{enumerate}}
	\newcommand{\een}{\end{enumerate}}

\theoremstyle{plain}
\newtheorem{theorem}{Theorem}[section]

\newtheorem{lemma}{Lemma}[section]

\theoremstyle{definition}

\numberwithin{equation}{section}
\begin{document}
	\title[On Induced Matching of Graphs]{On Induced Matching numbers of stacked-book graphs} 
	
	
	\author[T.C Adefokun]{Tayo Charles Adefokun$^1$ }
	\address{$^1$Department of Computer and Mathematical Sciences,
		\newline \indent Crawford University,
		\newline \indent Nigeria}
	\email{tayoadefokun@crawforduniversity.edu.ng}
	
	\author[O.L. Ogundipe]{Opeoluwa Lawrence Ogundipe $^2$}
	\address{$^2$Department of Mathematics,
		\newline \indent University of Ibadan,
		\newline \indent Nigeria}
	\email{opeogundipe2002@gmail.com}

	\author[D.O.A Ajayi]{Deborah Olayide Ajayi$^2$}
	\address{$^2$Department of Mathematics,
		\newline \indent University of Ibadan,
		\newline \indent Nigeria}
	\email{adelaideajayi@yahoo.com}

	
	\begin{abstract} 
		Suppose that $G$ is a simple, undirected graph. An induced matching in $G$ is a set of edges $M$ in the edge set $E(G)$ of $G$ such that if $e_1, e_2 \in M$, then no endpoint $v_1, v_2$ of $e_1$ and $e_2$ respectively is incident to any edge $e_k \in  E(G)$ such that $e_k$ is incident to any edge in $M$. Denoted by $im (G)$, the maximum cardinal number of $M$ is known as the induced matching number of $G$. In this work, we probe $im(G)$ where $G = G_{m,n}$, which is the stacked-book graph obtained by the Cartesian product of the star graph $S_m$ and path $P_n$.  
	\end{abstract} 
	
	\maketitle		
	
	\section{Introduction}
Suppose that $G$ is a graph with $E(G)$ as the edge set of $G$ while $V(G)$ denotes the vertex set of $G$. Let $M$ be subset of $E(G)$ such that for every $e_1, e_2 \in M$ there is no such edge in $E(G)$ to which any of the end points of $e_1$ and $e_2$ are commonly adjacent. Maximum Induced matching (MIM) problem is the generalization of the older graph matching problem, and it was introduced in \cite{SV1}. 

Suppose that $M$ is the largest induced matching in $G$ then the cardinal number of $M$, denoted by $im(G)$ is called the maximum induced matching number of $G$. Many work has has been on this subject. It has attracted interest mostly because of it is theoretically interested and it has a number of direct applications. In \cite{SV1}, the authors described MIM problem as "risk free" marriage where married couples who are perfectly matched are identified. Its usefulness in cryptography is also obvious. Cameron in her earlier work \cite{C1} showed that even though the MIM problem is NP-complete for bipartite graphs, it is easier to resolve for chordal graphs. This was also confirmed for circular graph in \cite{GL1}. Golumbic and Lewenstein \cite{GL2} established that the a relationship between MIM number and redundancy number in graphs and also showed that the MIM problem is polynomial-time solvable for tree graphs, while \cite{C2} investigated the MIM problem in intersection graphs. 
Recent works on MIM problem include \cite{M1} where the MIM number was extensively probed for grids $G_{n,m}=P_n \Box P_m$, the Cartesian product of paths $P_n$ and $P_m$. For odd $nm$, a bound $im(G_{n,m}) \leq \lfloor \frac{nm+1}{4} \rfloor$ was obtained. The bound was tightened in \cite{AA2} and further in \cite{AA1}. In \cite{XT1} investigation was made into obtaining exact algorithm for MIM problem of graphs on $n-$vertices. 

In this work, we probe the maximum induced matching problem for stacked-book graph $G_{m,n}$ class which are graphs obtained from the Cartesian product of star graphs $S_m$ and paths $P_n$. The MIM numbers are obtained for the initial range of these graphs while lower bounds of MIM number are derived for the general class.

	\section{Definitions}
		To make this works self-contained, we give the following definitions, which we shall adopt in the course of the paper. Definitions that are not considered as general will be given at the point of application.
	
	The vertex set of graph $G$ is $V(G)$ and $M$ is a subset of $E(G)$, the edge set of $G$, and $M$ is the induced matching of $G$. A vertex $v \in V(G)$ is called saturated if $v \in V(G)$ and unsaturated if otherwise. A star graph $S_m$ contains a central vertex $v_1$ (except if specifically indicated otherwise) with $m-1$ leaves, which are all incident to $v_1$ as pendants. A path $P_n$ contains $n$ edges and $n-1$ paths, while a cycle $C_m$ contains $m$ vertices and $n$ edges. Supposed that $u$ and $v$ are members of $V(G)$, then $d(u,v)$ is a positive integer, which is the distance between $u$ and $v$ in $G$. A vertex $v \in V(G)$ is called unstaurable if by the virtue of its position, can not be saturated either because of its distance from a saturated vertex or it is a the the right distance but not adjacent to a vertex that can be saturated in other to form an edge in the induced matching. A saturable vertex therefore, is the opposite of an un saturable vertex. The diameter of a graph is the largest distance between any two vertices on a graph $u$ and $v$, demoted by $diam(G)$. The set $[a,b]$ denoted set of integers from $a$ to $b$ while $[a]$ is a shortened for for $[1,a]$. 
	
	\subsection{Structure of a Stacked-book graph} 
	The stacked-book graph is the Cartesian product $S_m \Box P_n$ of a star graph $S_m$ and path $P_n$. Structurally, a $S_m \Box P_n$ contains $n$ number of $S_m$ stars such that there exist the $E(G') \in E(S_m \Box P_n)$, where $E(G')=\left\lbrace v_iu_i: v_i \in V(S_m(i)); u_i \in V(S_m(i+1), i \in [n])\right\rbrace $. Clearly, $E(S_m \Box P_n)=E(G') \cup E(\cup^n_{i=1}S_m(i))$, where $S_m(i)$ is designated as the the $i$th $S_m$ star graph for all $1 \leq i \leq n$

	\subsection{Initial Results}
	The following results are obvious
	\begin{theorem}\label{thm1}
		Let $P_n$ be a path graph on $n$ vertices. Then, $im(P_n)= \lceil \frac{n-1}{3}\rceil$.
	\end{theorem}

	\begin{theorem} \label{thm2}
		Let $C_n$ be a circle graph on $n$ vertices. Then $im(C_n)= \lfloor  \frac{n}{3}\rfloor$.
	\end{theorem}
    \begin{theorem}\label{thm3}\cite{M1}
    	 Suppose that $G_{3,n}$ is a grid graph obtained by the Cartesian product $P_3 \Box P_n$, where $n$ is even or odd. Then for a positive integer $k$, 
    	
    	\begin{center}
    		
    		$im(P_3 \Box P_n) = \left\{
    		\begin{array}{ll}
    			\lceil \frac{3n}{4} \rceil  &  \mbox{if} \;\; n \; \mbox {is even}; \\
    			\frac{3(n-1)}{4} & \mbox{if} \;\; n=4k+1  \\
    			\frac{3(n-1)+2}{4} &  \mbox{if} \; \; n=4k+3\\

    		\end{array}
    		\right.$
    		
    	\end{center}
    	
    \end{theorem}

	\section{Result}
	Now we present the results we have obtained in this work. First we show a result on induced matching on star graph $S_m$.
	
	\begin{theorem}\label{thm4}
		Let $S_m$ be a given a star graph such that the central vertex is $v_1$ and it is adjacent to $m-1$ leaves. Then $im(S_m)=1$ 
	\end{theorem}
(The implication of this result is that every star contains at most one element in its induced matching set.)
\begin{proof}
	Let $S_m$ be a start with $v_1$ being the central vertex. Then $v_1$ is saturated. Suppose that $v_k \in V(S_m)$, $k \leq m$, such that $v_k$ is saturated. Then $v_1v_k \in M$. Now for all $i$, $i \neq k$, $v_i \in V(G)$ is unsaturated since the $diam(S_m)=2$. Thus $im(S_m)=|m|=1$.
\end{proof}
Now we present our first results on the induced matching of stacked-book graph $G_{m,n}$.

\begin{lemma}\label{lem1}
	Suppose that $G_{m,n}$ is a stacked-book graph. Then if the induced matching number of $G_{m,n}$ is obtained, then the central vertices $v_1(1)$, $v_1(2)$ of factor stars $S_m(1)$ and $S_m(2)$ of $G_{m,n}$ are not unsaturated. 
\end{lemma}

\begin{proof}
	Suppose that $v_1$ and $u_1$  are the central vertices of $S_m(1)$ and $S_m(2)$ stars. Now, suppose that $v_1$ is saturated. Then either $v_1v_i \in M$, $v_i \in V(S_m(1))$ for some $2 \leq i < m $ or $v_1u_1 \in M$. Suppose that $v_1 v_i \in M$. By Theorem \ref{thm4}, if $v_1$ is saturated, then at least $m-2$ vertices on $S_m$ will be unsaturated. Thus, for all $v_i \in V(S_m(1))$, $v_iu_i \notin M$. Same argument holds if $u_1 \in S_m(2)$ is saturated. Thus, $im(G_{m,2}) = 1$. Now, suppose that $v_1u_1 \in M$. Since $d(v_1,v_i)=1=d(u_1,u_i)$ for all $i \in [2,m]$, then $v_i,u_i$ are unsaturated for all $i \in [2,m]$.  Clearly there exists a path $P_5$ in $G_{m,2}$. From Theorem \ref{thm1}, $P_5$ contains two edges in $M$ of $G_{m,2}$. Thus a contradiction. 
\end{proof}
Now we present the first theorem,
\begin{theorem}
	Let $G_{m,2}$ be a stacked-book graph. Then $im(G_{m,2})=m-1$. 
\end{theorem}

\begin{proof}
	Let $G_{m,2}$ be a stacked-book graph. Then there exist $S_m(1), S_m(2) \subseteq G_{m,2}$ with vertices $v_1, v_2 \cdots v_m$ and $u_1, u_2, \cdots u_m$ and a path $P_5(i) = v_i \rightarrow u_i \rightarrow u_1 \rightarrow u_{i+1} \rightarrow v_{i+1}$, for all $i \in [2,m]$. Thus, there exits, the set $\bar{P}=P_5(2), P_5(3), \cdots , P_5(\frac{m-1}{2})$, if $m$ is odd. Thus, there are $\frac{m-1}{2}$ number of $P_5-$paths. Now, by Theorem \ref{thm1}, $im(P_5)= 2$. Clearly, $\bar{P}$ consists of all the edges in $E(G_{m,2})$ that can be in $M$. Therefore, $im(G_{m.2}) \leq 2\left(\frac{m-1}{2} \right)=m-1$. Suppose that $m$ is even. Then, set $P*=\left\lbrace P_5(2), P_5(3), \cdots , P_5(\frac{m-2}{2}), P_3(t) \right\rbrace $, where $P_3(t)= v_k \rightarrow u_k \rightarrow u_1$. So, $im(P* \backslash P_3(t))=2 \left(\frac{m-2}{2}\right)=m-2$. By an earlier result, $im(P_3(t))=1$. Therefore, $im(P*)=m-1$. Hence, for any integer $m$, $im(G_{m,2}) \leq m-1$. Conversely, by definition of induced matching and stacked-book graph, $v_2u_2, v_3u_3, \cdots, v_mu_m$, satisfying the distance conditions to belong to $M$. Thus, $im(G_{m,2}) \geq m-1$ and hence the claim.   
\end{proof}

Next we consider the induced matching in $G_{m,3}$, where $m$ is either even or odd and show that the graph contains the same induced matching as $G_{m,2}$.

\begin{theorem}\label{thm5}
	Let $G_{m,3}$ be a stacked-book graph. Then $im(G_{m,3})=m-1$. 
\end{theorem}

To proof Theorem \ref{thm5}, we need two results, the first one, which is about about the nature of induced matching and distances between vertices of graphs, is more like a folklore because it follows from the definitions of induced matching of graphs.

\begin{lemma}\label{lem2}
	Let $e_1$ be a member of $M$ of a graph $G$. Then some edge $e_2 \in E(G)$ also belongs to $M$if there exists $v_1 \in e_1$ and $u_1 \in e_2$ such that $d(v_1,u_1)\geq 3$ and $v_2 \in e_1$ and $u_2 \in e_2 $, such that $d(v_2,u_2) \geq 2$.
\end{lemma}
\begin{proof}
	The proof follows from the definition of induced matching $M$ of graph $G$
\end{proof}

\begin{lemma} \label{lem3}
	Let $G_{m,3}$ be a stacked-book graph with factor star graphs $S_m(1)$, $S_m(2)$ and $S_m(3)$ such that $v_1 \rightarrow u_1 \rightarrow w_1$ is a $P_3$ path in $G_{m,3}$, where $v_1, u_1$ and $w_1$ are the central vertices of the respective factor star graphs.Then if $u_1$ is saturated, and $u_1v_k \in M$ for some $v_k \in V(G_{m,3})$, then $|M|=1$ and thus, not the maximum induced matching of $G_{m,3}$.
\end{lemma}

\begin{proof}
	For $v_k \in V(G_{m,3})$, $v_k \neq u_1$, for which $v_i \in G_{m,3}$ such that $d(v_k,v_i)=3$ since the $diam(G_{m,3})= 3$. However, suppose that $v_iv_j \in E(G_{m,3})$, for which $d(v_k,v_i)=3$. It is clear that $v_i$ is a leaf if some $S_m(t)$, $t \in \left\lbrace 1,3\right\rbrace $. Thus, $d(u_1,v_j) =1$, hence a contradiction to Lemma \ref{lem2} and hence the result. 
\end{proof}	

\subsection{Proof of Theorem \ref{thm5}} Now we proceed to proof Theorem \ref{thm5}.

\begin{proof} Suppose that $|M| > m-1$. Let $v_1,u_1$ and $w_1$ be the central vertices of $S_m(1), S_m(2)$ and $S_m(3)$ respectively. Clearly, $v_1u_1, u_1w_1 \notin M$ from Lemma \ref{lem3}. Now, first we show that $v_1$ is not saturable. Suppose that $v_1$ is saturable, then $v_1v_q \in M$, where $v_q$ is a leaf on $S_m(1)$. By an earlier result, subgraph induced $S_m(1)$ and $S_m(2)$ does not contain another member of $M$. Also, let $v_qu_q \in E(G_{m,3})$, with $u_q \in S_m(2)$ and $u_qw_q \in E(G)$, with $w_q \in S_m(3)$. By earlier result, $u_qw_q \notin M$. In like manner, if $w_1$ is saturated, and $w_1w_q \in M$ no other edge in subgraph of $G_{m,3}$ induced by $S_m(2)$ and $S_m(3)$ is a member of $M$, and $v_qu_q \notin M$. Without loss of generality, suppose that $v_1,v_q \in M$, then only $\bar{M}=\left\lbrace u_iw_i : i \in [2,m]; \; i \neq\right\rbrace \subset E(G_{m,3}) $ will be member of $M$. Thus $|\bar{M}|=m-2$ and so $|M|=m-1$, which is a contradiction. Now it has been established that none of the pendants of $S_m(1), S_m(2)$ and $S_m(3)$ can be in $M$. Thus, the possible members of $M$ are $\left\lbrace v_iu_i: i \in [2,m]\right\rbrace  \cup \left\lbrace u_iw_i: i \in [2,m] \right\rbrace = M' $. Clearly, $|\bar{M}|=2(m-1)$. By Lemma \ref{lem2}, only half of the members of $\bar{M}$ can be in $M$. Thus, $im(G_{m,3}) \leq m-1$. Reasonably, $im(G_{m,2}) \leq im(G_{m,3})$. By earlier result, therefore, $iu(G_{m,3}) \geq m-1$ and thus $im(G_{m,3}) = m-1$.
	
\end{proof}

Next we investigate the induced matching number of $G_{m,4}$. We start with a lemma that will be employed in the main result.

\begin{lemma} \label{lem4}
	Let $G_{m,4}$ be a stacked-book graph such that $S_m(1), S_m(2), S_m(3)$ and $S_m(4)$ are the factor stars of $G_{m,4}$. Suppose that $im(G_{m,4}) \geq m$. Then if $M' = \left\lbrace u_iw_i: i[2,m]; \; u_i \in S_m(2), w_i \in S_m(3) \right\rbrace $, then $M'$ is not a subset of $M$.
\end{lemma}

\begin{proof}
	It is eay to see that $|M'| = m-1$. Now, suppose that $M' \subset M$, then $u_i, w_i$ are saturated for all $i \in [2,m]$. Thus, no vertex $v_i \in S_m(1)$ and $r_i \in S_m(4)$ is saturable, for $i \in [2,m]$, which implies that $im(G_{m,4})=m-1$ and thus, a contradiction.
\end{proof}

Nex we consider the main theorem.

\begin{theorem}\label{thm6}
	Let $S_m(1), S_m(2), S_m(3)$ and $S_m(4)$ be the factor star graphs of the stacked-book graph $G_{m,4}$. Then, $im(G_{m,4})=m$.
\end{theorem}

\begin{proof}
	By Lemma \ref{lem4}, suppose that at least some edge in \\$M' = \left\lbrace u_iw_i: i[2,m]; \; u_i\in S_m(2), w_i \in S_m(3) \right\rbrace $ is not in $M$. Suppose therefore that $u_kw_k \in M$. Then for $v_k\in S_m(1)$, and $r_k \in S_m(4)$, $v_1v_k, r_1r_k \in M$, where $v_1$ and $r_1$ are the central vertices of $S_m(1)$ and $S_m(4)$ respectively. Thus, $im(G_{m,4}) \geq m$. Conversely, suppose that $im(G_{m,4})=m+1$. Now, let $u_1, w_1$ be the central vertices of $S_m(2)$ and $S_m(3)$ respectively. Suppose that one of $u_i,w_i$, say $u_i$ is saturated such that $u_1u_i \in M$. Then, from earlier result, no edge in the subgraph of $G_{m,4}$ induced by $S_m(1)$, $S_m(2)$ and $S_m(3)$ is contained in $M$. Likewise, if $w_1w_i \in M,$ then all other vertices on the subgraph of $G_{m,4}$ induced by $S_m(2)$, $S_m(3)$ and $S_m(4)$  are unstaurable. If any of the pendant of $S_m(2)$ and $S_m(3)$ is in $M$, then $M=2$. Now, note as well that if $u_1w_1 \in M$, then by the distances of $u_1$ and $w_1$ to the rest of vertices on $S_m(1), S_m(2), S_m(3)$ and $S_m(4)$, only $u_1w_1$ will be in $M$. Thus for optimal $M$, some members of $M''= \left\lbrace v_iu_i; i \in [2,m]\right\rbrace$ or $M'''= \left\lbrace w_ir_i: i \in[2,m] \right\rbrace $ will have to be in $M$. 
	
	Now clearly, it can be see that $|M' \cup M''|=2(m-1)$ and only $m-1$ members of $M' \cup M''$ can be in $M$. Based on this observable fact, at least there will exist a $w_i \in S_m(3)$ that is not saturable. Thus, there exist a saturable vertex $r_i \in S_m(4)$, such that $r_1r_i \in M$. By earlier result, there is no other pendant of $S_m(4)$ that is in $M$. Thus, $im(G_{m,4}) < m+1$ and hence a contradiction. Therefore, $im(G_{m,4}) \leq m$ and the claim follows.

\end{proof}
Now we consider the case of $G_{m,5}$. We shall need some new results to aid the proof.

\begin{lemma}\label{lem5}
	Suppose that $w_1 \in S_m(3)$ is the central vertex of $S_m(3)$, where $\left\lbrace S_m(i): i \in [1,5]\right\rbrace $ is the set of factor stars of $G_{m,5}$. If $w_1$ is saturated, then for $M$ of $G_{m,5}$, $|M| \leq 2m-3$. 
\end{lemma}
Suppose that $w_1$ is the central vertex of $S_m(3)$ and it is saturated. Then one of the $w_1w_k, u_1w_1$ and $w_1r_1$ belongs to $M$ where $u_1, r_1$ are central vertices of $S_m(2)$ and $S_m(4)$respectively. Suppose that $w_1w_k \in M$, where $k \leq m$. Now for all $i \in [2,m], i \neq k$, $w_i \in S_m(3)$ is unsaturable by earlier results. Thus members of $\left\lbrace u_iw_i:i \in [2,m] \right\rbrace $ and $\left\lbrace w_ir_i: r_i \in S_M(4), i \in [2,m] \right\rbrace $  do not belong in $M$. Also it clear to see that both edges $v_ku_k, r_kt_k \notin M$, where $t_k \in S_m(5)$. Now, from earlier technique, it can be deduced that $v_1v_i, t_1t_i  \notin M$ for all $i \in [2,m]$. Thus, only $E'=\left\lbrace v_iu_i: i \in [2,m], i \neq k \right\rbrace $ and $E'' = \left\lbrace r_it_i: i \in [2,m], i \neq k \right\rbrace $ can be in $M$. Clearly, $|E' \cup E''|=2(m-2)$. Thus $|M|=2m-3$. Also, if $u_1w_i \in M$, it can be seen by following the definitions of induced matching that no other edges in the subgraph of $G_{m,5}$ induced by $S_m(1), S_m(2)$ and $S_m(3)$ is a member of $M$ and from earlier results, only $m-1$ edges of the subgraph of $G_{m,5}$ induced by $S_m(3), S_m(4)$ and $S_m(5)$ can be in $M$.  Thus, $M$ consists of at most $m$ edges, which is not more that $2m-3$, since $m \geq 3$. 
\begin{lemma}\label{lem6}
	Suppose that $im(G_{m,5}) \geq 2(m-1)$. Then $u_1, w_1$ and $r_1$, the central vertices of $S_m(2), S_m(3)$ and $S_m(4)$ respectively are unsaturated. 
\end{lemma}

\begin{proof}
	Proof follows from last theorem and an earlier result. 
\end{proof}

Now we proceed to the probe the induced matching of $G_{m,5}$.

\begin{theorem}
	Let $G_{m,5}$ be a stacked-book graph. Then, $im(G_{m,5}) = 2(m-1)$.
\end{theorem}

\begin{proof}
	From the last results, we see that if $u_1, w_1, r_1 $ are unsaturated, then $|M| \geq 2m-3$. Now we show that $im(G_{m,5}) \geq 2(m-1)$. Note that there exists a path $P_5(i) = v_i \rightarrow u_i \rightarrow w_i \rightarrow r_i \rightarrow t_i$, for all $i \in [2,m]$. Therefore, there are $m-1$ such paths in $G_{m,5}$. From earlier results, $im(P_5)=2$. Thus, $im(G_{m,5})\geq 2(m-1)$. Conversely, $u_1, w_1, r_1$ are established not to be saturated for the claim to hold. The edges in $E(G_{m,5})$ left to be members of $M$ the pendants of $S_m(1)$ and $S_m(5)$ and the paths $P_5(i)$ defined earlier. Suppose that a pendant each from $S_m(1)$ and $S_m(5)$ belong to $M$, then by definition of induced matching, at most one edge on each of the paths $P_5(i)$ can be a member of $M$. Thus $|M|=m+1$. The only alternative is if no pendant of $S_m(1)$ and $S_m(5)$ is a member of $M$. Thus, at most two edges on each member of $P_5(i)$ will be in $M$. Thus, $|M| \leq 2(m-1)$ and so, $im(G_{m,5})=2(m-1)$.
\end{proof}

Now we generalize the results. 

\begin{theorem}
	Let $G_{m,n}$ be a stacked-book graph such that $n$ is even. Then
		\begin{center}
		
		$im(G_{m,n}) \geq \left\{
		\begin{array}{ll}
			m\lceil \frac{n}{4} \rceil -1  &  \mbox{if} \;\; n \equiv 2 \mod 4; \\
			\frac{mn}{4} & \mbox{if} \;\; n \equiv 0 \mod 4.

		\end{array}
		\right.$
		
	\end{center}
	
\end{theorem}
\begin{proof}
	the claims follow from combining the results earlier proved where n are even numbers.
\end{proof}
\begin{theorem}
	Let $G_{m,n}$ be stacked-book graph with $n$ odd. Then
	
	\begin{center}
		
		$im(G_{m,n}) \geq \left\{
		\begin{array}{ll}
			m\lfloor \frac{n}{4} \rfloor +2  &  \mbox{if} \;\; n \equiv 3 \mod 4; \\
			\frac{mn+3m-8}{4} & \mbox{if} \;\; n \equiv 1 \mod 4.

		\end{array}
		\right.$
		
	\end{center}
\end{theorem}

We have established the lower bound for the the MIM numbers for the stacked-book graphs. From our preliminary work into establishing the tighter bounds, we have reasons to suggest that the results in the last two theorems may coincide with the upper bounds, and thus we come up with the conjectures below.

\begin{conj}
		Let $G_{m,n}$ be a stacked-book graph such that $n$ is even. Then
	\begin{center}
		
		$im(G_{m,n}) = \left\{
		\begin{array}{ll}
			m\lceil \frac{n}{4} \rceil -1  &  \mbox{if} \;\; n \equiv 2 \mod 4; \\
			\frac{mn}{4} & \mbox{if} \;\; n \equiv 0 \mod 4.

		\end{array}
		\right.$
		
	\end{center}
	
\end{conj}
\begin{conj}
	Let $G_{m,n}$ be stacked-book graph with $n$ odd. Then
	
	\begin{center}
		
		$im(G_{m,n}) = \left\{
		\begin{array}{ll}
			m\lfloor \frac{n}{4} \rfloor +2  &  \mbox{if} \;\; n \equiv 3 \mod 4; \\
			\frac{mn+3m-8}{4} & \mbox{if} \;\; n \equiv 1 \mod 4.

		\end{array}
		\right.$
		
	\end{center}
\end{conj}

\section{Conclusion}

We have obtained the MIM number of stacked-book graphs $G_{m,n}$ for all $m$ and for $n \in [1,5]$. These results are building blocks for obtaining the lower bounds for the cases where $n \geq 6$. The conjecture at the end of the work suggests that the lower bounds obtained in this work will in fact be equal to the upper bounds if those can be found. It must be noted that finding the lower bounds or the MIM numbers for the complete stacked-book graphs class will take rigorous effort and therefore may worth considering as a new task.

\end{document}